\documentclass[12pt]{article}
\usepackage{amsmath}
\usepackage{amsthm}
\usepackage{amssymb}

\newtheorem{thm}{Theorem}
\newtheorem{cor}{Corollary}
\newtheorem{prop}{Proposition}

\newtheorem{ex}{Example}

\usepackage{tikz}
\usetikzlibrary{decorations.markings, arrows}
\usepackage{xparse}

\begin{document}

\vspace*{30px}

\begin{center}\Large
\textbf{Computing Hyperfibonacci Numbers by Means of Matrix Transformations and Jordan Forms}
\bigskip\large

\bigskip

Petra Marija De Micheli Vitturi\\
Faculty of Electrical Engineering, Mechanical Engineering and Naval Architecture, 
University of Split,\\ Split, Croatia\\
{\tt pgojun@fesb.hr}
\bigskip

Nevena Jakovčević Stor\\
Faculty of Electrical Engineering, Mechanical Engineering and Naval Architecture, 
University of Split,\\ Split, Croatia\\
{\tt nevena@fesb.hr}


\end{center}


\begin{abstract} 
The Hyperfibonacci sequence of the $r$th generation is defined recursively as a generalization of Fibonacci numbers, where each term is obtained by summing the terms of the Hyperfibonacci sequence of the preceding generation. We introduce the transformation matrix for Hyperfibonacci numbers, which enables us to determine the next term in a given generation. We explore the algebraic structure of that matrix, and its power of $n$, similarity transformations between these matrices and their Jordan canonical forms.  Finally, we analyze the powers of these matrices using their Jordan forms, obtaining compact and elegant formulas for expressing  $r$-generation Hyperfibonacci numbers in terms of Fibonacci numbers. 
\end{abstract}


\section{Introduction}

Hyperfibonacci numbers were introduced by A. Dil and I. Mez\"{o} in \cite{DillM} as follows: 

For nonnegative integers $n$ and $r$, the Hyperfibonacci numbers of the $r$th generation, denoted by $F_n^{(r)}$, are given by recurrence relation
\begin{equation}\label{F-def}
F_n^{(r)}= \sum \limits_{k=0}^{n} F_k^{(r-1)} \text{, for } r>0 ,
\end{equation}
with initial conditions $F_n^{(0)}= F_n$, for $n\geq 0$, and $F_0^{(r)}=0$, $F_1^{(r)}=1$, for $r\geq 0$.

Directly from the definition, one can find the relation
\begin{equation}\label{F-rekurzija}
F_{n}^{(r)} = F_{n-1}^{(r)} + F_{n}^{(r-1)}. 
\end{equation}

Using relation  \eqref{F-def}, we give Hyperfibonacci numbers $F_n^{(r)}$, for $n=0,1, \dots, 9$ and $r=0,1, \dots , 5$. The results are listed in Table \ref{F_n^r}.

\begin{table}[h]
\centering
\begin{tabular}{ c|r r r r r r r r r r r r} 

$r \backslash n$ & $0$ & $1$ & $2$ & $3$ & $4$ & $5$ &  $6$ & $7$ & $8$  & $9$  & $\cdots$  \\ 
 \hline 
 $0$  & 0 & 1 & 1 & 2 & 3 & 5 & 8 & 13 & 21 & 34 & $ \cdots$ \\
 $1$ & 0 &  1 & 2 & 4 & 7 & 12 & 20 & 33 & 54 & 88 & $ \cdots$ \\
$2$  &  0 & 1 & 3 & 7 & 14 & 26 & 46 & 79  & 133 & 221 & $\cdots$ \\
 $3$ &  0 & 1 & 4 & 11 & 25 & 51 & 97 & 176 & 309 & 530  & $\cdots $ \\
 $4$ &  0 & 1 & 5 & 16 & 41 & 92 & 189 & 365 & 674 & 1204 & $\cdots$ \\
 $5$ &  0 & 1 & 6 &  22 & 63 & 155  &  344 & 709 & 1383 & 2587 & $\cdots$ \\
$ \vdots$ & $ \vdots$ & $ \vdots$ & $ \vdots$ & $ \vdots$ & $ \vdots$ & $ \vdots$ & $ \vdots$ & $ \vdots$ & $ \vdots$ & $ \vdots$ \\
 \end{tabular}
\caption{The Hyperfibonacci numbers} \label{F_n^r}
\end{table}

L. L. Christea et al. derived several additional identities \cite{CMU}, including the following,
\begin{equation*}
F_{n}^{(1)}= \frac{\varphi^{n+2} - \overline{\varphi}^{n+2}}{\sqrt{5}} - 1 =  F_{n+2} - 1 ,
\end{equation*}
\begin{equation*}
F_{n}^{(2)}= \frac{\varphi^{n+4} - \overline{\varphi}^{n+4}}{\sqrt{5}} - n-3 = F_{n+4} -n - 3 .
\end{equation*}
Furthermore, 
\begin{equation*}
\lim_{n \to \infty} \frac{F_{n+1}^{(r)}}{F_{n}^{(r)}} = \varphi,
\end{equation*}
where $\varphi$  and $\overline{\varphi}$ are the golden ratio and its conjugate, respectively, i.e.
$$\varphi =\frac{1+ \sqrt{5}}{2} , \, \,  \overline{\varphi} =\frac{1- \sqrt{5}}{2} .$$

Combinatorial interpretation of these numbers by means of tilings is known \cite{BB}. A. T. Benjamin et al. provided a combinatorial interpretation for some determinant identities involving Fibonacci numbers \cite{BCQ}. The various generalizations of Fibonacci numbers have been studied in the literature \cite{LM, EPM}. Furthermore, E. Karaduman investigated the determinants of the matrices obtained by sequences of some generalization of Fibonacci numbers \cite{EK}. As for Fibonacci numbers, it is known that these can be generated from the following matrix representation,
\begin{equation}\label{Fmatrix}
\begin{bmatrix}
F_{n+1} & F_n \\
F_n & F_{n-1}
\end{bmatrix} = \begin{bmatrix}
1  &1 \\
1 & 0
\end{bmatrix}^n .
\end{equation}
Taking the determinants of both sides, one can obtain the Cassini identity,
$$F_{n+1}F_{n-1} - F_{n}^2 = (-1)^n .$$

Being motivated by these results, in what follows, we study matrix transformation between Hyperfibonacci numbers in Section 2. We show that the powers of that matrix can be used for generating Hyperfibonacci numbers of the $r$th generation.  In Section 3, we employ Jordan decompositions of such matrices to derive representations of Hyperfibonacci numbers of the $r$th generation in terms of Fibonacci numbers.

\section{The Transformation Matrix}

In this section, we give the transformation matrix, which, when applied to the column vector of Hyperfibonacci numbers, generates the next vector in the same generations.

Let $A_r$ be a  $(r+2) \times (r+2)$ matrix of the form
\begin{equation*}
A_r := \begin{bmatrix}
1 & 1 & 0 &    &  \cdots  &   &  &  & 0 \\
0 & 1 & 1 & 0 &    & \cdots  &  &   & 0 \\
0 & 0 & 1 & 1 & 0 &  &  \cdots&   & 0 \\
 &  &  &  &  &  &  &   & \\
\vdots &  &  &  & \ddots &  &  &  & \vdots \\
 &  &  &  &  &  &  &   & \\
0 & & \cdots& &   & 0 & 1 & 1 & 0 \\
0 &  & \cdots&  & & 0 & 0 & 1 & 1 \\
0 & &\cdots &  & & 0 & 0 & 1 & 0 
\end{bmatrix}.
\end{equation*}
The matrix $A_r$ is defined by
$$
(A_r)_{i,i} = (A_r)_{i,i+1} = 1, \quad 1 \le i \le r+1,
$$
and
$$
(A_r)_{r+2,\,r+1} = 1,
$$
with all remaining entries equal to zero.

\begin{prop}
For a positive integer $n$ and a nonnegative integer $r$, the following holds
\begin{equation*}
A_r \cdot 
 \begin{bmatrix}
F_{n-1}^{(r)} \\
F_{n }^{(r-1)} \\
F_{n+1}^{(r-2)} \\
 \\
 \vdots \\
 \\
F_{n+r-1}^{(1)} \\
F_{n+r} \\
F_{n+r-1}
\end{bmatrix} 
=
 \begin{bmatrix}
F_{n}^{(r)} \\
F_{n+1 }^{(r-1)} \\
F_{n+2}^{(r-2)} \\
 \\
 \vdots \\
 \\
F_{n+r}^{(1)} \\
F_{n+r+1} \\
F_{n+r}
\end{bmatrix} .
\end{equation*}
\end{prop}

\begin{proof}

The proof is obtained by multiplication and by using the relation  (\ref{F-rekurzija}). We use the following notation,
\begin{equation*}
\Phi_{n, r}= 
 \begin{bmatrix}
F_{n}^{(r)} \\
F_{n+1 }^{(r-1)} \\
F_{n+2}^{(r-2)} \\
 \\
 \vdots \\
 \\
F_{n+r}^{(1)} \\
F_{n+r+1} \\
F_{n+r}
\end{bmatrix} .
\end{equation*}
Multiplying the $i$th row of the matrix $A_r$ by $\Phi_{n-1,r}$ and using the relation \eqref{F-rekurzija}, give
$$ F_{n+i-2}^{(r-i+1)} + F_{n+i-1}^{(r-i)} = F_{n+i-1}^{(r-i+1)}, $$
for $i= 1, \dots ,r$.
When multiplying the $(r+1)$th row of the matrix $A_r$ by $\Phi_{n-1,r}$
we obtain $$ F_{n+r} + F_{n+r-1} = F_{n+r+1}.$$ 
Multiplication of the last row of the matrix $A_r$ with $\Phi_{n-1,r}$
equals the last element of the matrix $\Phi_{n,r}$, which completes the proof.
\end{proof}

For $r=0$, the matrix $A_r$ has the form
 \begin{equation*}
A_0=\begin{bmatrix}
1  &1 \\
1 & 0
\end{bmatrix} .
\end{equation*}
Raising the matrix $A_0$ to the power of $n$, as we mention before in the relation \eqref{Fmatrix}, we get
\begin{equation*}
A_0^n= \begin{bmatrix}
F_{n+1} & F_n \\
F_n & F_{n-1}
\end{bmatrix}. 
\end{equation*}
Moreover, the first few powers of the matrix $A_1$ are the following
\begin{equation*}
A_1=\begin{bmatrix}
1 & 1 & 0 \\
0 & 1 & 1 \\
0 & 1 & 0 
\end{bmatrix} = \begin{bmatrix}
1 & F_1^{(1)} & F_0^{(1)}  \\
0 & F_2 & F_1 \\
0 & F_1 & F_0 
\end{bmatrix} ,
\end{equation*}
\begin{equation*}
A_1^2=\begin{bmatrix}
1 & 1 & 0 \\
0 & 1 & 1 \\
0 & 1 & 0 
\end{bmatrix}  \begin{bmatrix}
1 & 1 & 0  \\
0 & 1 & 1 \\
0 & 1 & 0 
\end{bmatrix}=\begin{bmatrix}
1 & 2 & 1 \\
0 & 2 & 1 \\
0 & 1 & 1 
\end{bmatrix} = \begin{bmatrix}
1 & F_2^{(1)} & F_1^{(1)}  \\
0 & F_3 & F_2 \\
0 & F_2 & F_1 
\end{bmatrix} ,
\end{equation*}

\begin{equation*}
A_1^3=\begin{bmatrix}
1 & 1 & 0 \\
0 & 1 & 1 \\
0 & 1 & 0 
\end{bmatrix}  \begin{bmatrix}
1 & 2 & 1 \\
0 & 2 & 1 \\
0 & 1 & 1  
\end{bmatrix}=\begin{bmatrix}
1 & 4 & 2 \\
0 & 3 & 2 \\
0 & 2 & 1 
\end{bmatrix} = \begin{bmatrix}
1 & F_3^{(1)} & F_2^{(1)}  \\
0 & F_4 & F_3 \\
0 & F_3 & F_2
\end{bmatrix} .
\end{equation*}

In the following proposition, we present a general form for the powers of $n$ of the matrix $A_r$.

\begin{prop}
For a positive integer $n$ and a nonnegative integer $r$, $n\geq r$, the powers of $n$ of the matrix $A_r$ are
\begin{equation*}
A_r^n = \begin{bmatrix}
{n \choose 0 } & {n \choose 1}  & {n \choose 2 }  &    &  \cdots  &  & & {n \choose r-1 }  & F_{n-r+1}^{(r)} & F_{n-r}^{(r)} \\
0 & {n \choose 0 }  & {n \choose 1 }  & {n \choose 2 }  &    & \cdots  & & {n \choose r-2 }  & F_{n-r+2}^{(r-1)}  & F_{n-r+1}^{(r-1)} \\
0 & 0 & {n \choose 0 }  & {n \choose 1 }  & {n \choose 2 }  & & \cdots   & {n \choose r-3 } &  F_{n-r+3}^{(r-2)} & F_{n-r+2}^{(r-2)} \\
 &  &  &  &  &  &  &  & & \\
  &  &  &  &  &  &  & &  & \\
\vdots &  &  & & \ddots  & &  &  & \vdots & \vdots \\
 &  &  &  &  &  &  & &  & \\
  &  &  &  &  &  &  & &  & \\
0 & & \cdots& &  & & 0 & {n \choose 0 }  & F_n^{(1)} & F_{n-1}^{(1)} \\
0 &  & \cdots&  & & & 0 & 0 & F_{n+1} & F_n \\
0 & &\cdots &  & & & 0 & 0 & F_n & F_{n-1}
\end{bmatrix}.
\end{equation*}
\end{prop}

\begin{proof}
Taking into account the previous examples, we continue with the proof by induction.
%

Let us assume that matrix $A_r^{n}$ has the required form for $n\geq r$. Now, we shall prove the statement for the matrix $A_r^{n+1}=A_r \cdot A_r^n$.

The first $r$ rows of the matrix $A_r$, when multiplying the first $r$ columns of the matrix $A_r^n$, form an upper-triangular $r \times r$ block with binomial coefficients. The diagonal entries are all 1, 
$${n \choose 0} = 1 = {n+1 \choose 0}, $$ 
while the other nonzero entries are obtained as sums of two binomial coefficients with the same upper index and consecutive lower indices. More precisely, if $i$th row of $A_r$ is multiplied by $j$th column of  $A_r^n$, $1\leq i < j \leq r$, then we have $${n \choose j-i-1} + {n \choose j-i-2} = {n+1 \choose j-i-1}. $$

Multiplying two last rows $A_r$ by first $r$ columns of $A_r^n$, yield zero matrix, 
$$(A_r \, A_r^n)_{r+1:r+2,1:r} = \mathbf{0} ,$$
where $\mathbf{0} $ denotes the zero matrix of the corresponding size.

It remains to check the multiplication by the last two columns of the matrix $A_r^n$. Multiplying it with the first $r$ rows of the matrix $A_r$ and  by using \eqref{F-rekurzija} we have
$$
(A_r^n)_{i,r+1}=F_{n-r+i}^{(r-i+1)} + F_{n-r+i+1}^{(r-i)} = F_{n-r+i+1}^{(r-i+1)},
$$
and
$$
(A_r^n)_{i,r+2}=F_{n-r+i-1}^{(r-i+1)} + F_{n-r+i}^{(r-i)} = F_{n-r+i}^{(r-i+1)},
$$
for $1 \le i \le r$. Finally, the last two rows of matrix $A_r$ multiplied by the last two columns of matrix $A_r^n$ give the corresponding Fibonacci numbers. 

\end{proof}

\begin{ex}
One can find that the $10$th power of $A_3$ is
\begin{equation*}
A_3^{10}=\begin{bmatrix}
1 & 10 & 45 & 309 & 176 \\
0 & 1 & 10 & 221 & 133 \\
0 & 0 & 1 & 143 & 88 \\
0 & 0 & 0 & 89 & 55 \\
0 & 0 & 0 & 55 & 34
\end{bmatrix} = \begin{bmatrix}
{10 \choose 0} & {10 \choose 1} & {10 \choose 2}   & F_{8}^{(3)} & F_{7}^{(3)}  \\
0 & {10 \choose 0} & {10 \choose 1}  & F_{9}^{(2)} & F_{8}^{(2)} \\
0 & 0 & {10 \choose 0} & F_{10}^{(1)} & F_9^{(1)} \\
0 & 0 & 0 & F_{11} & F_{10} \\
0 & 0 & 0 & F_{10} & F_{9}
\end{bmatrix} .
\end{equation*}
\end{ex}

\section{Jordan Form}

Motivated by the occurrence of Hyperfibonacci numbers of the $r$th generation in the powers of the matrix
$A_r$, we proceed to explore identities arising from matrix decompositions and from the structural properties of $A_r$.

It should be noted that the matrix $A_r$ is diagonalizable only for $r = 0,1 $.

\begin{equation*}
A_0= \begin{bmatrix}
\overline{\varphi}  & \varphi \\
1 & 1
\end{bmatrix} 
 \begin{bmatrix}
\overline{\varphi}  &0 \\
0 & \varphi
\end{bmatrix} 
 \begin{bmatrix}
\frac{-1}{\sqrt{5}} &\frac{\varphi}{\sqrt{5}}\\
\frac{1}{\sqrt{5}}  & -\frac{\overline{\varphi}}{\sqrt{5}}
\end{bmatrix} ,
\end{equation*}

\begin{equation*}
A_1= \begin{bmatrix}
1 & \frac{ \overline{\varphi }}{ -\varphi }  &  \frac{ \varphi }{ -\overline{\varphi} }  \\
0 & \overline{\varphi}  & \varphi \\
0 & 1 & 1
\end{bmatrix} 
 \begin{bmatrix}
 1 & 0 & 0 \\
0 & \overline{\varphi}  &0 \\
0 & 0 & \varphi
\end{bmatrix} 
 \begin{bmatrix}
 1 & 1 & 1 \\
0 & \frac{-1}{\sqrt{5}} & \frac{\varphi}{\sqrt{5}} \\
0 & \frac{1}{\sqrt{5}}  & -\frac{\overline{\varphi}}{\sqrt{5}} 
\end{bmatrix} .
\end{equation*}
Since the matrix $A_r$ is not diagonalizable for $r \ge 2$, we consider its Jordan canonical decomposition (see e.g. \cite{HJ},\cite{GVL}).

Let us recall, for $A \in \mathbb{C}^{\ n\times n}$, there exists an invertible matrix $P \in \mathbb{C}^{n \times n}$ such that
\[
A = P J P^{-1},
\]
where $J$ is a \emph{Jordan matrix} (or \emph{Jordan canonical form}) of $A$. 
The matrix $J$ is block-diagonal, consisting of Jordan blocks $J_k(\lambda)$, where each
\[
J_k(\lambda) =
\begin{pmatrix}
\lambda & 1 & 0 & \cdots & 0 \\
0 & \lambda & 1 & \cdots & 0 \\
\vdots & \ddots & \ddots & \ddots & \vdots \\
0 & \cdots & 0 & \lambda & 1 \\
0 & \cdots & \cdots & 0 & \lambda
\end{pmatrix} \in \mathbb{C}^{k \times k},
\]
corresponds to an eigenvalue $\lambda$ of $A$. 

Thus,
\[
J = \operatorname{diag}\big(J_{k_1}(\lambda_1), J_{k_2}(\lambda_2), \dots, J_{k_m}(\lambda_m)\big),
\]
where $m$ is the number of distinct eigenvalues and $k_i$ is the multiplicity of  eigenvalue $\lambda_i$, $1 \le i \le m$.


In the following proposition, we present the Jordan form of the matrix $A_r$.

\begin{prop}
For a nonnegative integer $r$, Jordan form of matrix $A_r$ is
$$A_r= P_r J_r P_r^{-1} ,$$
where  $$P_r = \begin{bmatrix}
1 & 0 & 0 &    &  \cdots  &   &  &  \frac{ \overline{\varphi }}{ (-\varphi)^{r} } & \frac{ \varphi }{ (-\overline{\varphi})^{r} } \\
0 & 1 & 0 & 0 &    & \cdots  &  & \frac{ \overline{\varphi }}{ (-\varphi)^{r-1} }  & \frac{ \varphi }{ (-\overline{\varphi})^{r-1} } \\
0 & 0 & 1 & 0 & 0 &  &  \cdots& \frac{ \overline{\varphi }}{ (-\varphi)^{r-2} }   &  \frac{ \varphi }{ (-\overline{\varphi})^{r-2} }  \\
 &  &  &  &  &  &  &   & \\
\vdots &  &  &  & \ddots &  &  &  & \vdots \\
 &  &  &  &  &  &  &   & \\
0 & & \cdots& &   & 0 & 1 & \frac{ \overline{\varphi }}{ -\varphi }  & \frac{ \varphi }{ -\overline{\varphi} } \\
0 &  & \cdots&  & & 0 & 0 & \overline{\varphi } & \varphi \\
0 & &\cdots &  & & 0 & 0 & 1 & 1 
\end{bmatrix},$$
$$J_r=\begin{bmatrix}
1 & 1 & 0 &    &  \cdots  &   &  &  & 0 \\
0 & 1 & 1 & 0 &    & \cdots  &  &   & 0 \\
0 & 0 & 1 & 1 & 0 &  &  \cdots&   & 0 \\
 &  &  &  &  &  &  &   & \\
\vdots &  &  &  & \ddots &  &  &  & \vdots \\
 &  &  &  &  &  &  &   & \\
0 & & \cdots& &   & 0 & 1 & 0 & 0 \\
0 &  & \cdots&  & & 0 & 0 & \overline{\varphi} & 0 \\
0 & &\cdots &  & & 0 & 0 & 0 & \varphi 
\end{bmatrix},$$ and 
$$P_r^{-1} = \begin{bmatrix}
1 & 0 & 0 &    &  \cdots  &   &  & -F_{r+1} & -F_r \\
0 & 1 & 0 & 0 &    & \cdots  &  & -F_r  & -F_{r-1} \\
0 & 0 & 1 & 0 & 0 &  &  \cdots&  -F_{r-1} & -F_{r-2} \\
 &  &  &  &  &  &  &   & \\
\vdots &  &  &  & \ddots &  &  &  & \vdots \\
 &  &  &  &  &  &  &   & \\
0 & & \cdots& &   & 0 & 1 & -F_2 & -F_1 \\
0 &  & \cdots&  & & 0 & 0 & - \frac{1}{\sqrt{5}} & \frac{\varphi}{\sqrt{5}} \\
0 & &\cdots &  & & 0 & 0 & \frac{1}{\sqrt{5}} & -\frac{ \overline{\varphi} }{\sqrt{5}}
\end{bmatrix}.$$

\end{prop}

\begin{proof}

By solving the equation
\begin{equation*}
\det (A_r - \lambda I ) = 0 ,
\end{equation*}
we obtain the eigenvalues of the matrix $A_r$, 
$$\lambda_1=1, \, \lambda_2= \overline{\varphi},\text{ and } \lambda_3= \varphi,$$ where the eigenvalue $\lambda_1 = 1$ has multiplicity $r$, while the remaining two eigenvalues each have multiplicity $1$. Therefore, matrix $J_r$ consists of three Jordan blocks, $J_r(1)$, $J_1( \overline{\varphi})$ and $J_1(\varphi)$.

Furthermore, by solving the equation
\begin{equation*}
A_r v = \lambda v,
\end{equation*}
we obtain the corresponding eigenvectors, respectively, to the above eigenvalues
$$v_1=(1,0, \dots, 0),$$
$$v_2= \left(  \frac{ \overline{\varphi }}{ (-\varphi)^{r} } ,  \frac{ \overline{\varphi }}{ (-\varphi)^{r-1} } , \dots,  \frac{ \overline{\varphi }}{ -\varphi } ,  \overline{\varphi }, 1 \right),$$ 
and 
$$v_3=\left(\frac{ \varphi }{ (-\overline{\varphi})^{r}}, \frac{ \varphi }{ (-\overline{\varphi})^{r-1}}, \dots, \frac{ \varphi }{ -\overline{\varphi} }, \varphi , 1  \right).$$
This determines the form of the matrix $P_r$.

We now proceed to prove that $P_r^{-1}$ indeed represents the inverse matrix of $P_r$. The multiplication of the first $r$ rows of $P_r$ with the first $r$ columns of $P_r^{-1}$ yields an $r \times r$ block with ones on the diagonal and zeros elsewhere. Moreover, the multiplication of the first $r$ rows of $P_r$ with the last two columns of $P_r^{-1}$ results in the zero matrix. In particular, the multiplication of the
$i$th row, $1\leq i \leq r$, of  $P_r$ and $(r+1)$th column of $P_r^{-1}$  gives
\begin{align*}
\begin{split}
 & -F_{r+1-i+1} +  \frac{ \overline{\varphi}}{ (-\varphi)^{r-i+1} }\cdot \frac{-1}{\sqrt{5}} +  \frac{ \varphi }{ (-\overline{\varphi})^{r-i+1} }\cdot \frac{1}{\sqrt{5}}  \\[4pt]
 & = -F_{r-i+2} + \frac{\varphi (-\varphi)^{r-i+1}  - \overline{\varphi} (-\overline{\varphi})^{r-i+1} }{\sqrt{5} \, \varphi^{r-i+1}  \,\overline{\varphi}^{r-i+1} } \\[4pt]
 & = -F_{r-i+2} + \frac{(-1)^{r-i+1}(\varphi^{r-i+2}- \overline{\varphi}^{r-i+2})}{\sqrt{5} \, (-1)^{r-i+1}} \\
 & = -F_{r-i+2} + F_{r-i+2} \\
 & = 0 .
\end{split}
\end{align*}
In the same manner, multiplication by the last column of $P_r^{-1}$ also results in zero.

It is evident that multiplying the last two rows of the matrix $P_r$ with first $r$  columns of the matrix $P_r^{-1}$ results in the zero matrix, and that the product of the last two rows of the matrix $P_r$ and the last two columns of the matrix $P_r^{-1}$ yields the $2 \times 2$ identity matrix. This completes the proof.
\end{proof}

It is known that if $$A_r= P_r J_r P_r^{-1}$$ then  $$A_r^n= P_r J_r^n P_r^{-1}, \, n\geq 1 .$$

One can find that the $n$th power of $J_r$ has the following form

\begin{equation*}
J_r^n = \begin{bmatrix}
{n \choose 0 } & {n \choose 1}  & {n \choose 2 }  &    &  \cdots  &  & & {n \choose r-1 }  & 0 & 0 \\
0 & {n \choose 0 }  & {n \choose 1 }  & {n \choose 1 }  &    & \cdots  & & {n \choose r-2 }  & 0  & 0 \\
0 & 0 & {n \choose 0 }  & {n \choose 1 }  & {n \choose 2 }  & & \cdots   & {n \choose r-3 } &  0 & 0  \\
 &  &  &  &  &  &  &  & & \\
  &  &  &  &  &  &  & &  & \\
\vdots &  &  & & \ddots  & &  &  & \vdots & \vdots \\
 &  &  &  &  &  &  & &  & \\
  &  &  &  &  &  &  & &  & \\
0 & & \cdots& &  & & 0 & {n \choose 0 }  & 0 & 0\\
0 &  & \cdots&  & & & 0 & 0 &  \overline{\varphi}^{n} & 0 \\
0 & &\cdots &  & & & 0 & 0 & 0 & \varphi^{n}
\end{bmatrix}.
\end{equation*}

The idea is that by using the Jordan decomposition, we can express the matrix 
$A_r^n$ in a form that allows us to compute the $r$-th generation Hyperfibonacci numbers in terms of the Fibonacci numbers.

In matrix $A_r^n$ on first row and $(r+1)$th column  is Hyperfibonacci number $F_{n-r+1}^{(r)}$.
\begin{equation}\label{position}
\bigl(A_r^n\bigr)_{1,\,r+1} = F_{\,n-r+1}^{(r)}.
\end{equation}

Now we compute the element on the same position in the matrix $ P_r J_r^n P_r^ {-1}$.  Let us first multiply $ P_r J_r^n$,
\begin{equation*}
P_r J_r^n=\begin{bmatrix}
{n \choose 0 } & {n \choose 1}  & {n \choose 2 }  &    &  \cdots  &  & & {n \choose r-1 }  &  \frac{ \overline{\varphi}^{n+1}  }{ (-\varphi)^{r} }&  \frac{ \varphi^{n+1} }{ (-\overline{\varphi})^{r} } \\
0 & {n \choose 0 }  & {n \choose 1 }  & {n \choose 1 }  &    & \cdots  & & {n \choose r-2 }  & \frac{ \overline{\varphi}^{n+1}  }{ (-\varphi)^{r-1} } & \frac{ \varphi^{n+1} }{ (-\overline{\varphi})^{r-1} }  \\
0 & 0 & {n \choose 0 }  & {n \choose 1 }  & {n \choose 2 }  & & \cdots   & {n \choose r-3 } &  \frac{ \overline{\varphi}^{n+1}  }{ (-\varphi)^{r-2} } & \frac{ \varphi^{n+1} }{ (-\overline{\varphi})^{r-2} }  \\
 &  &  &  &  &  &  &  & & \\
  &  &  &  &  &  &  & &  & \\
\vdots &  &  & & \ddots  & &  &  & \vdots & \vdots \\
 &  &  &  &  &  &  & &  & \\
  &  &  &  &  &  &  & &  & \\
0 & & \cdots& &  & & 0 & {n \choose 0 }  & \frac{ \overline{\varphi}^{n+1}  }{ (-\varphi)}  & \frac{ \varphi^{n+1} }{ -\overline{\varphi} }  \\
0 &  & \cdots&  & & & 0 & 0 &  \overline{\varphi}^{n+1}  & \varphi^{n+1}  \\
0 & &\cdots &  & & & 0 & 0 &  \overline{\varphi}^{n}  & \varphi^{n}  
\end{bmatrix}.
\end{equation*}
Then we compute $A_r^n= P_r J_r^n P_r^{-1}$. On the position $ (P_r J_r^n P_r^{-1})_{1, r+1}$ we get
\begin{align*}
\begin{split}
& - \sum \limits_{k=0}^{r-1} {n \choose k } F_{r-k+1} +  \frac{ \overline{\varphi}^{n+1}  }{ (-\varphi)^{r} } \frac{-1}{\sqrt{5}}    +  \frac{ \varphi^{n+1} }{ (-\overline{\varphi})^{r} }  \frac{1}{\sqrt{5}}  \\[4pt]
&= - \sum \limits_{k=0}^{r-1} {n \choose k } F_{r-k+1}  + \frac{ \varphi^{n+r+1}-\overline{\varphi}^{n+r+1}}{\sqrt{5}} \\[4pt] 
&= - \sum \limits_{k=0}^{r-1} {n \choose k } F_{r-k+1}  + F_{n+r+1}.
\end{split}
\end{align*}

Which together with \eqref{position} gives
$$F_{n-r+1}^{(r)}= - \sum \limits_{k=0}^{r-1} {n \choose k } F_{r-k+1}  + F_{n+r+1} . $$
Finally, by substituting $n-r+1 \rightarrow n$ we have 
$$F_n^{(r)}=  F_{n+2r} - \sum \limits_{k=0}^{r-1}  {n+r-1 \choose k} F_{r-k+1 } .$$
This proves the following theorem.

\begin{thm}\label{tm3}
For a nonnegative integer $n$ and a positive integer $r$, the Hyperfibonacci numbers satisfy
\begin{equation*}
F_n^{(r)}=  F_{n+2r} - \sum \limits_{k=0}^{r-1} {n+r-1 \choose k} F_{r-k+1 }  .
\end{equation*}
\end{thm}

\begin{ex}
 Evaluation of the Theorem \ref{tm3} identity for \( n = 8 \) and \( r = 3 \) gives
\begin{align*}
F_8^{(3)} 
&= F_{8 + 2\cdot3} - \sum_{k=0}^{2} {8 + 3 - 1 \choose k}  F_{3 - k + 1} \\[4pt]
&= F_{14} - \sum_{k=0}^{2} {10 \choose k} F_{4 - k} .
\end{align*}
Expanding the summation yields
$$\sum_{k=0}^{2} {10 \choose k} F_{4 - k} 
= {10 \choose 0} F_4  + {10 \choose 1} F_3 + {10 \choose 2} F_2 .$$
Since Fibonacci numbers in question are \( F_2 = 1 \), \( F_3 = 2 \), and \( F_4 = 3 \), we find
$$3\cdot 1 + 2\cdot 10 + 1\cdot 45= 3 + 20 + 45 = 68.$$
Therefore,
$$F_8^{(3)} = F_{14} - 68.$$
Since \( F_{14} = 377 \), it follows that
$$F_8^{(3)} = 377 - 68 = 309.$$
The computed result for $F_8^{(3)}$ is equal to the value given in the Table \ref{F_n^r}.
\end{ex}

Special case of Theorem \ref{tm3} for $n=0$, yields the following Corollary \ref{cor}.
\begin{cor} \label{cor}
For a nonnegative integer $r$, Fibonacci numbers satisfy
\begin{equation*}
F_{2r}=  \sum \limits_{k=0}^{r-1} {r-1 \choose k} F_{r-k+1 } .
\end{equation*}
\end{cor}

 \section{Conclusion}
We have presented a matrix-based approach to compute Hyperfibonacci numbers of any generation $r$, using transformation matrices and their Jordan forms. This framework provides compact formulas linking Hyperfibonacci numbers to Fibonacci numbers and reveals the underlying algebraic structure. Our results offer an efficient computational method and a foundation for further exploration.



\end{document}